\documentclass[12pt]{article}
\usepackage{amsmath}
\usepackage{amsfonts}
\usepackage{amsthm}

\def \im {\cong}

\begin{document}

\begingroup
\newtheorem{theorem}{Theorem}[section]  
\newtheorem{cor}[theorem]{Corollary}    
\newtheorem{lem}[theorem]{Lemma}        
\newtheorem{prop}[theorem]{Proposition} 
\newtheorem{conj}[theorem]{Conjecture}  
\endgroup

\newtheorem{construction}{Construction}[section] 

\newtheorem{example}{Example}        
\newtheorem{fig}{Figure}          

\newtheorem{definition}[theorem]{Definition}  
\newtheorem{rem}[theorem]{Remark}  
\newtheorem{notation}{Notation.}
\renewcommand{\thenotation}{}
\newtheorem{terminology}{Terminology.}
\renewcommand{\theterminology}{}


\title{The uniqueness of a distance-regular graph 
with intersection array $\{32,27,8,1;1,4,27,32\}$ and
related results}

\author{Leonard H. Soicher\\
School of Mathematical Sciences\\
Queen Mary University of London\\
Mile End Road, London  E1 4NS, UK\\
{\small L.H.Soicher@qmul.ac.uk}
}
\maketitle

\begin{abstract}
It is known that, up to isomorphism, there is a
unique distance-regular graph $\Delta$ 
with intersection array $\{32,27;1,12\}$
(equivalently, $\Delta$ is the unique strongly regular
graph with parameters $(105,32,4,12)$).
Here we investigate the distance-regular antipodal covers of $\Delta$.
We show that, up to
isomorphism, there is just one distance-regular antipodal triple cover
of $\Delta$ (a graph $\hat\Delta$ discovered by the author over twenty
years ago), proving that there is a unique distance-regular graph with
intersection array $\{32,27,8,1;1,4,27,32\}$. In the process, we confirm
an unpublished result of Steve Linton that there is no distance-regular
antipodal double cover of $\Delta$, and so no distance-regular graph
with intersection array $\{32,27,6,1;1,6,27,32\}$. We also show there
is no distance-regular antipodal $4$-cover of $\Delta$, and so no
distance-regular graph with intersection array $\{32,27,9,1;1,3,27,32\}$,
and that there is no distance-regular antipodal $6$-cover of $\Delta$
that is a double cover of $\hat\Delta$.

\bigskip
[Keywords:
distance-regular graph;
strongly regular graph;
antipodal cover;
fundamental group]

\end{abstract}

\vfill\eject

\section{Introduction}

A graph $\Gamma$ is \textit{distance-regular} with \textit{intersection array}
\[ \{b_0,b_1,\ldots,b_{d-1};c_1,c_2,\ldots,c_d\} \] 
if $\Gamma$ is connected with diameter $d$, and for $i=0,1,\ldots,d$,
whenever vertices $v,w$ are at distance $i$ in $\Gamma$, there are
exactly $b_i$ vertices adjacent to $w$ at distance $i+1$ from $v$
and exactly $c_i$ vertices adjacent to $w$ at distance $i-1$ from $v$
(with the convention that $b_d=c_0=0$). Distance-regular graphs occur
in many areas of discrete mathematics, including coding theory, design
theory and finite geometry. 
See, for example, the encyclopaedic reference \cite{BCN89} by Brouwer,
Cohen and Neumaier. 
The present paper is a contribution to the problem
of determining the distance-regular
graphs with a given intersection array. 

We shall investigate the antipodal 
distance-regular covers of the Goethals-Seidel graph \cite{GS70}, 
the unique (up to isomorphism) distance-regular graph $\Delta$
with intersection array $\{32,27;1,12\}$
\cite{Co06,DC08}. The graph $\Delta$ can be constructed
as the second subconstituent of the second subconstituent
of the famous McLaughlin graph, the unique distance-regular
graph with intersection array $\{112,81;1,56\}$ 
(see \cite{GS70} and \cite{DC08}).   
 
A distance-regular antipodal triple cover $\hat\Delta$ of $\Delta$ was
constructed by the author in \cite{So93}, 
but its uniqueness was not determined. 
The main purpose of this paper is to prove that, up to isomorphism, 
$\hat\Delta$ is the unique distance-regular
graph with intersection array $$\{32,27,8,1;1,4,27,32\}.$$ In the
process of doing this, we confirm an unpublished result of Steve Linton
which shows that there is no distance-regular graph with intersection array 
$\{32,27,6,1;1,6,27,32\}$.  We also prove there is no
distance-regular graph with intersection array $\{32,27,9,1;1,3,27,32\}$,
and that there is no distance-regular antipodal $6$-cover of $\Delta$
that is a double cover of the distance-regular triple cover $\hat\Delta$. 

We classify distance-regular antipodal $r$-covers
of $\Delta$ by studying the $r$-fold topological covers of the
$2$-dimensional simplicial complex whose $0$-, $1$-, and $2$-simplices
are respectively the vertices, edges, and triangles of $\Delta$. Our
main tool is version~2.0 of the author's \textsf{GAP} program 
described in \cite{RS00} for the computation of fundamental groups,
certain quotients of fundamental groups, and covers of finite abstract
$2$-dimensional simplicial complexes. This program is freely available from
\cite{fundamental}, where we also provide a \textsf{GAP/GRAPE} \cite{GAP,GRAPE}
logfile of all the computations described in this paper. It is hoped that 
\cite{fundamental} and the
methods of this paper will be useful in further classifications 
of covers of graphs.

All graphs in this paper are finite and undirected, with no loops
and no multiple edges. Throughout, we follow
\cite{BCN89} for graph-theoretical concepts and notation.
An important new reference for 
distance-regular graphs, covering developments since 
\cite{BCN89} was published, is van~Dam, Koolen and Tanaka \cite{vDKT14}.
See also Brouwer \cite{Br}. 
A good reference for the group-theoretical 
concepts used in this paper is Robinson \cite{Ro95}.
We denote the commutator subgroup of a group $G$ 
by $[G,G]$, the cyclic group of order $n$ is denoted by $C_n$,
and where $p$ is a prime, the elementary
abelian group of order $p^k$ is denoted by $C_p^k$. 

\section{The fundamental group and covers}

Let $\Gamma$ be a non-empty connected graph. 
We also consider $\Gamma$ to be an abstract 
$2$-dimensional simplicial complex (or $2$-complex) 
whose $0$-, $1$-, and $2$-simplices are respectively the vertices, edges,
and triangles of $\Gamma$. With this in mind, thoughout this paper, 
by a \textit{cover} of $\Gamma$, we mean  
a pair $(\tilde\Gamma,\theta)$, where $\tilde\Gamma$ is a graph and 
$\theta:V(\tilde\Gamma)\to V(\Gamma)$ is a surjection, called
a \textit{covering map}, such that the following hold:
\begin{itemize}
\item
for every $v\in V(\Gamma)$, the \textit{fibre} $v\theta^{-1}$ of $v$ is
a coclique (independent set) of $\tilde\Gamma$, 
\item
the union of any two distinct fibres mapping under $\theta$ to a non-edge of $\Gamma$ 
is a coclique of $\tilde\Gamma$,
\item
the induced subgraph on any two
fibres mapping under $\theta$ to an edge of $\Gamma$ 
is a perfect matching in $\tilde\Gamma$,
\item
the induced subgraph on any three fibres mapping under $\theta$ to a triangle of
$\Gamma$ consists of pairwise disjoint triangles in $\tilde\Gamma$.
\end{itemize}
Note that, as defined here, $\tilde\Gamma$ is a cover of $\Gamma$ precisely
when $\tilde\Gamma$ is a topological cover of $\Gamma$ when they are both
viewed as $2$-complexes as described above. 

When we do not need to specify the covering map explicity, we may
denote a cover $(\tilde\Gamma,\theta)$ simply by $\tilde\Gamma$.
If each fibre of a cover $\tilde\Gamma$
of $\Gamma$ has the same positive integer cardinality $r$, then we call $\tilde\Gamma$ an
$r$-\textit{cover} of $\Gamma$. A $2$-cover is also called a \textit{double cover}, 
and a $3$-cover is also called a \textit{triple cover}.  If $\Gamma$ is a non-complete
graph, then an \textit{antipodal} $r$-cover $\tilde\Gamma$ of $\Gamma$
is a connected $r$-cover of $\Gamma$ with the property that being equal
or at maximum distance in $\tilde\Gamma$ is an equivalence relation on
$V(\tilde\Gamma)$, whose equivalence classes are the fibres.

The connected $r$-covers of $\Gamma$ correspond to the transitive
permutation representations of degree $r$ of the fundamental group $G$
of $\Gamma$ (regarded as a $2$-complex), defined with respect to a fixed,
but arbitrary, spanning tree of $\Gamma$. We shall explain this further in
what follows. For a more general and detailed exposition, see \cite{RS00},
on which our explanation is based. We note that the choice of fixed spanning
tree does not affect the isomorphism class of the fundamental group of $\Gamma$. 

Now, fix a spanning tree $T$ of $\Gamma$. Then, 
for each arc (ordered edge) $(v,w)$ of $\Gamma$, define an abstract 
group generator $g_{v,w}$. Then $G$ is the 
finitely presented group whose generators are these $g_{v,w}$
and whose relations are: 
\begin{itemize} 
\item
$g_{s,t}=1$ for each arc $(s,t)$ of $T$;
\item 
$g_{v,w}g_{w,v}=1$ for each edge $\{v,w\}$ of $\Gamma$;
\item 
$g_{x,y}g_{y,z}g_{z,x}=1$ for each triangle $\{x,y,z\}$ of $\Gamma$.
\end{itemize}

Suppose $(\tilde\Gamma,\theta)$ is any connected 
$r$-cover of $\Gamma$. Then $T\theta^{-1}$ consists of $r$
disjoint copies $T_1,\ldots,T_r$ of $T$, with $T_i\theta=T$.  We may
thus label the vertices of $\tilde\Gamma$ by ordered pairs $(v,i)$,
with $v\in V(\Gamma)$ and $i\in I=\{1,\ldots,r\}$, so that the fibre
of a vertex $v$ of $\Gamma$ is $\{(v,1),\ldots,(v,r)\}$ and $(v,i) \in
V(T_i)$ for each $i$. In particular, note that a different choice for the
fixed spanning tree $T$ would only result in a relabelling of vertices within the
fibres of $\tilde\Gamma$.  

Now for each arc $(v,w)$ of $\Gamma$, let
$\rho_{v,w}$ be the permutation of $I$ defined by $i\rho_{v,w} = j$
if and only if $\{(v,i),(w,j)\}$ is an edge of $\tilde{\Gamma}$.
If $\{x,y,z\}$ is a triangle of $\Gamma$, then $\{x,y,z\}
\theta^{-1}$ is a disjoint union of triangles in $\tilde{\Gamma}$, and
so $\rho_{x,y}\rho_{y,z}\rho_{z,x}$ is the identity permutation. 
It now follows that the map $\rho$ defined on the generators of
$G$ by $(g_{v,w})\rho = \rho_{v,w}$ 
extends to a transitive permutation representation from $G$
to the symmetric group $S_r$. The $r$-cover $\tilde{\Gamma}$ is completely
defined by this representation.

Conversely, every transitive permutation representation $\rho:G\to S_r$
defines a connected $r$-cover, denoted $\Gamma_\rho$, of $\Gamma$. 
For each $v\in V(\Gamma)$, the fibre of $v$ in $\Gamma_\rho$ consists
of the ordered pairs $(v,i)$, for $i\in I$, and $(v,i)$ and $(w,j)$ are joined by an
edge in $\tilde\Gamma$ 
if and only if $\{v,w\}$ is an edge of $\Gamma$ and $j=i((g_{v,w})\rho)$.
The preimage in $G$ of the stabilizer in $G\rho$ of a point $i\in I$ is
isomorphic to the fundamental group of $\Gamma_\rho$ (see \cite{RS00}).

We consider two covers $(\Gamma_1,\theta_1)$ and $(\Gamma_2,\theta_2)$ of $\Gamma$ to be
\textit{isomorphic} if there is a graph isomorphism from $\Gamma_1$
to $\Gamma_2$ which maps the fibres of $\Gamma_1$ to those of
$\Gamma_2$. Thus, isomorphic
covers differ only by a relabelling of the fibres and of the vertices
within each fibre.  In particular, every connected $r$-cover of $\Gamma$
is isomorphic to a cover of the form $\Gamma_\rho$, for some transitive
permutation representation $\rho:G\to S_r$, corresponding to the
preimage in $G$ of the stabilizer in $G\rho$ of the point $1$.

Isomorphism of covers can be checked as follows using \textsf{nauty}
\cite{NAUTY}, which can be called from within \textsf{GRAPE}
\cite{GRAPE}.  Given the cover $\Gamma_1$ of $\Gamma$, we make a
$\{\textrm{red},\textrm{blue}\}$-vertex-coloured graph $\Gamma_1^+$.
The red-coloured vertices of $\Gamma_1^+$ are 
the vertices of $\Gamma_1$, with two red
vertices joined in $\Gamma_1^+$ precisely when those vertices are joined 
in $\Gamma_1$.  The blue-coloured vertices of $\Gamma_1^+$
are in one-to-one correspondence with the fibres of $\Gamma_1$, and a blue vertex is
joined in $\Gamma_1^+$ only to the red vertices in the corresponding fibre. Similarly,
we make $\Gamma_2^+$ from the cover $\Gamma_2$ of $\Gamma$. Then
$\Gamma_1$ and $\Gamma_2$ are isomorphic as covers of $\Gamma$
if and only if $\Gamma_1^+$
is isomorphic to $\Gamma_2^+$ by a colour-preserving graph isomorphism.

We are usually interested in classifying covers up to isomorphism.  First note
that if two graphs are isomorphic, then every cover of the first is
isomorphic to some cover of the second.  Moreover, if $(\Gamma_1,\theta_1)$
and $(\Gamma_2,\theta_2)$ are isomorphic covers of $\Gamma$, and
$(\tilde\Gamma_1,\tilde\theta_1)$ is any cover of $\Gamma_1$, then
$(\tilde\Gamma_1,\tilde\theta_1\theta_1)$ is a cover of $\Gamma$
isomorphic to $(\tilde\Gamma_2,\tilde\theta_2\theta_2)$ for some cover
$(\tilde\Gamma_2,\tilde\theta_2)$ of $\Gamma_2$.

\section{Imprimitivity and covers}

We may sometimes be interested in a cover of a cover. 

Let $\Gamma$ be a connected graph, let $(\Gamma_1,\theta_1)$ be a
connected $m$-cover of $\Gamma$, and let $(\Gamma_2,\theta_2)$
be a connected $n$-cover of $\Gamma_1$.  Then clearly,
$(\Gamma_2,\theta_2\theta_1)$ is a connected $mn$-cover of $\Gamma$. Given
a spanning tree $T$ of $\Gamma$, we take a spanning tree $T_1$
of $\Gamma_1$ containing the forest $T\theta_1^{-1}$, and define
the fundamental group $G$ of $\Gamma$ with respect to $T$ and
the fundamental group $G_1$ of $\Gamma_1$ with respect to $T_1$.
Then $(\Gamma_1,\theta_1)$ corresponds to a transitive representation
$\rho_1:G\to S_m$, with the fibre of $v$ in $\Gamma_1$ being labelled
$\{(v,i):1\le i\le m\}$, as described previously, and $(\Gamma_2,\theta_2)$
corresponds to a transitive representation $\rho_2:G_1\to S_n$, with
the fibre of $(v,i)$ being labelled $\{(v,i,j):1\le j\le n\}$. Now
$(\Gamma_2,\theta_2\theta_1)$ corresponds to a transitive representation
$\rho:G\to S_{mn}$, and if $m,n>1$, then $G\rho$ acts imprimitively on
the indices $(i,j)$ of the fibre of a vertex of $\Gamma$, the blocks of
imprimitivity being $\{(i,j):1\le j\le n\}$, for $i=1,\ldots,m$.

Conversely, suppose that $m,n>1$ and $\rho:G\to S_{mn}$ is a transitive
permutation representation of the fundamental group $G$ of $\Gamma$ such
that $G\rho$ is an imprimitive group having $m$ blocks of size $n$. Then,
if $\sigma:G\to S_m$ is the transitive permutation action of $G\rho$ on
the blocks of imprimitivity, we see that $\Gamma_\rho$ is an $n$-cover
of the $m$-cover $\Gamma_\sigma$ of $\Gamma$.

We have the following: 

\begin{theorem}
\label{TH:impdr}
Let $\Gamma$ be a non-complete distance-regular graph and suppose 
$\Gamma_\rho$ is
a distance-regular antipodal $mn$-cover of $\Gamma$ corresponding to a
transitive permutation representation $\rho$ of the fundamental group
$G$ of $\Gamma$,
such that $G\rho$ has $m>1$ blocks of imprimitivity
of size $n>1$. Then $\Gamma_\rho$ must be an $n$-cover of an antipodal
distance-regular $m$-cover of $\Gamma$. 
\end{theorem}

\begin{proof}
As above, we may suppose that $G\rho$ is a group of permutations of
$\Omega:=\{(i,j):1\le i\le m,\,1\le j\le n\}$, with blocks of 
imprimitivity $B_i:=\{(i,j):1\le j\le n\}$, for $i=1,\ldots,m$, and that  
the fibre of $\Gamma_\rho$ mapping to the vertex $v$ of $\Gamma$
is labelled as $\{(v,i,j):1\le i\le m,\,1\le j\le n\}$. 

Now consider the partition
\[ \pi:=\{\,\{(v,i,j):1\le j\le n\}:v\in V(\Gamma),\,1\le i\le m\}\]
of $V(\Gamma_\rho)$. Then each part in $\pi$ is contained in a fibre of
$\Gamma_\rho$, and $\pi$ is an equitable (also called regular) partition of $V(\Gamma_\rho)$
(since each vertex in the part $\{(v,i,j):1\le j\le n\}$ is joined to exactly one or no
vertex in the part $\{(w,k,\ell):1\le \ell\le n\}$, with a join to one 
vertex precisely when $\{v,w\}\in E(\Gamma)$ and $\rho_{v,w}$ maps the block 
$B_i$ to $B_k$).  We may thus apply Theorem~6.2 of Godsil and Hensel 
\cite{GH92} to deduce that
$\Gamma_\rho/\pi$ is an antipodal distance-regular $m$-cover of $\Gamma$
(see also Theorem~7.3 in \cite[Section~11]{Go93}, attributed to 
Brouwer, Cohen and Neumaier). 
\end{proof}
 
\section{On the distance-regular antipodal $r$-covers of $\Delta$}

Throughout this section, $\Delta$ denotes the unique distance-regular graph
with intersection array $\{32,27;1,12\}$ (equivalently, the unique strongly regular
graph with parameters $(105,32,4,12)$).

Suppose $\Gamma$ is a distance-regular antipodal $r$-cover of  
$\Delta$. Parameter feasibility conditions (see \cite{BCN89}) imply 
$\Gamma$ has diameter $4$, $r\in\{2,3,4,6\}$, and  
$\Gamma$ has intersection array
\begin{equation}
\{32,27,12(r-1)/r,1;1,12/r,27,32\}.\label{intarray}
\end{equation} 
Conversely, if $\Gamma$ is a distance-regular graph with intersection array
(\ref{intarray}), then $\Gamma$
is an antipodal $r$-cover of a distance-regular graph
with intersection array $\{32,27;1,12\}$ (see \cite{Ga74}), so $\Gamma$ is 
an antipodal distance-regular $r$-cover of (a graph isomorphic to) $\Delta$.

To study covers of $\Delta$, we explore quotients of the 
fundamental group of $\Delta$, viewed as a $2$-complex, 
with respect to a fixed spanning tree $T$ of $\Delta$. 
This fundamental group is denoted throughout this section by $D$. 
It was shown in \cite{So93} that $D$ infinite. 

\begin{theorem}
\label{TH:fund}
The abelianised fundamental group $D/[D,D]$ of $\Delta$ 
is isomorphic to $C_2^{16}\times C_3^2$.
\end{theorem}

\begin{proof} We compute a presentation for $D/[D,D]$ using 
the program \cite{fundamental}, and determine the abelian invariants 
of this group using \textsf{GAP} \cite{GAP}. 
\end{proof}

We are now in a position to classify the connected double covers of $\Delta$, which was done independently 
by Steve Linton over 20 years ago, using his vector enumeration algorithm \cite{Li93}.

\begin{theorem} 
\label{TH:2cov}
Up to isomorphism of covers, $\Delta$ has just $13$ connected double covers, with each
having its abelianised fundamental group isomorphic to $C_2^{15}\times
C_3^2$ , and none being distance-regular.  
\end{theorem}

\begin{proof}
Each subgroup of index $2$ in $D$ contains the commutator subgroup $[D,D]$.  
We compute the $2^{16}-1$ covers of $\Delta$ corresponding to the subgroups of
index $2$ in $D/[D,D]$, and using \textsf{GRAPE} calling \textsf{nauty},
we determine that, up to isomorphism of covers, there are
just $13$ such covers. We test each of these covers for distance-regularity
using \textsf{GRAPE}, and find that none of them are distance-regular.  
We use the program \cite{fundamental} to compute each of their abelianised
fundamental groups. 
\end{proof}

\begin{cor} There is no distance-regular graph with
intersection array $\{32,27,6,1;1,6,27,32\}$.
\end{cor}

We now consider the triple covers of $\Delta$, and start by showing that $D$
has no quotient isomorphic to the symmetric group $S_3$.
This is a corollary of the following:
 
\begin{prop} 
Let $N$ be a subgroup of $D$ of index $2$, 
and suppose $N$ has a subgroup $M$, such that $M$ is normal in $D$
and $N/M$ is abelian.  Then $D/M$ is abelian.  
\end{prop}

\begin{proof}
The subgroup $N$ of $D$ is isomorphic to the fundamental group of some 
connected $2$-cover of $\Delta$, so by 
Theorem~\ref{TH:2cov}, $[N,N]$ has index 
$2^{15}3^2$ in $N$, and so $[N,N]$ has index $2^{16}3^2$ in $D$, as does $[D,D]$,
by Theorem~\ref{TH:fund}. Thus $[N,N]=[D,D]$, and since 
$N/M$ is abelian, $M$ contains $[N,N]$, and so $D/M$ is abelian. 
\end{proof}

\begin{cor}
The group $D$ has no quotient isomorphic to the symmetric group $S_3$.
\end{cor}

We are now in a position to classify the connected triple covers of $\Delta$.  

\begin{theorem} 
\label{TH:3cov}
Up to isomorphism of covers, $\Delta$ has just two connected triple covers, $\Delta^*$ and $\hat\Delta$.
The cover $\Delta^*$ 
has abelianised fundamental group isomorphic to $C_2^{16}\times C_3^2$ and is not distance-regular.  
The cover $\hat\Delta$  has abelianised fundamental group isomorphic to $C_2^{18}\times C_3^2$ and
is distance-regular. 
\end{theorem}

\begin{proof}
Each connected triple cover of $\Delta$ is isomorphic to a cover
of the form $\Delta_\rho$, for some transitive permutation
representation $\rho:D\to S_3$. Since $D$ has no quotient isomorphic
to $S_3$, we must have $D\rho\cong C_3$, the cyclic group of
order $3$, and so each subgroup of index $3$ in $D$ is normal
and contains $[D,D]$.  

We compute the four covers of $\Delta$ corresponding to the subgroups of
index $3$ in $D/[D,D]$, and determine that, up to isomorphism of covers, 
there are just two such covers, $\Delta^*$ and $\hat\Delta$, and 
we calculate that these have the properties as stated in the theorem.  
\end{proof}

\begin{cor} Up to isomorphism, there is a unique distance-regular graph with
intersection array $\{32,27,8,1;1,4,27,32\}$.
\end{cor}

Further properties of this distance-regular graph $\hat\Delta$ are given in \cite{So93}
(where the graph is called $\Lambda$). 

We now show there is no distance-regular antipodal $4$-cover of $\Delta$. 

\begin{theorem} There is no distance-regular graph with
intersection array $\{32,27,9,1;1,3,27,32\}$.
\end{theorem}

\begin{proof}
A distance-regular graph with the intersection array of the 
theorem would be an antipodal $4$-cover of $\Delta$.
Such a $4$-cover would correspond to some transitive
permutation representation $\rho$ of degree $4$ of $D$.
The image $D\rho$ of $D$ cannot be imprimitive, for 
otherwise, by Theorem~\ref{TH:impdr}, $\Delta$ would have
a distance-regular antipodal $2$-cover. We cannot have 
$D\rho\cong S_4$, since $D$ has no quotient isomorphic to
$S_3$, which is a quotient of $S_4$. 
This leaves $D\rho\im A_4$ as the only possibility.

Suppose now $D$ has a quotient isomorphic to $A_4$.
Then there is a transitive permutation 
representation $\alpha:D\to S_6$, having three blocks of 
imprimitivity of size $2$, with $D\alpha\cong A_4$. The corresponding
cover $\Delta_\alpha$ is a $2$-cover of a
$3$-cover $\tilde\Delta$ of $\Delta$. 
The $3$-cover $\tilde\Delta$ corresponds to 
a normal subgroup $N$ of $D$ of index $3$,
and $N$ has a subgroup $M$ of index $4$ that is normal in $D$, 
such $D/M\cong A_4$. 

Suppose now $\tilde\Delta$ is isomorphic (as a cover of $\Delta$) to $\Delta^*$.
By Theorem~\ref{TH:3cov}, $[N,N]$ has index 
$2^{16}3^2$ in $N$, and so $[N,N]$ has index $2^{16}3^3$ in $D$. 
By Theorem~\ref{TH:fund}, $[D,D]$ has index $2^{16}3^2$ in $D$,
so $[N,N]$ is a normal subgroup of index $3$ in $[D,D]$. Since 
$N/M\im C_2^2$, $M$ contains $[N,N]$, and so either $D/M$ is abelian
or $D/M$ has a normal subgroup of order $3$, neither of which holds.   

Thus, we must have $\tilde\Delta$ isomorphic to $\hat\Delta$.  
We computed all $2^{18}-1$ $2$-covers of $\hat\Delta$
as covers of the form $\Delta_\rho$ of $\Delta$, for $\rho$ a
permutation representation from $D$ to $S_6$.  
Just three of these covers $\Delta_\rho$ have $D\rho$ 
isomorphic to $A_4$, corresponding to the three permutation isomorphic
representations of degree~$6$ of one quotient of $D$ isomorphic
to $A_4$. Now given a connected $6$-cover $\Delta_\rho$ with $D\rho\cong A_4$, 
we construct the cover $\Delta_\sigma$ of $\Delta$ defined by the
representation of $D\rho$ acting by right multiplication on
the four (right) cosets of a Sylow $3$-subgroup. Up to isomorphism of
covers of $\Delta$, there is only one such $\Delta_\sigma$, which we
find is not distance-regular. 
\end{proof}

Finally, we can prove the following: 

\begin{theorem}
There is no distance-regular antipodal $6$-cover of $\Delta$
that is a double cover of $\hat\Delta$. 
\end{theorem}

\begin{proof}
When we determined all $2^{18}-1$ connected $2$-covers of $\hat\Delta$,
we found that none is distance-regular. 
\end{proof}

This still leaves open the possibility of a distance-regular
antipodal $6$-cover of $\Delta$ corresponding to a primitive degree
$6$ permutation representation of $D$.
 
\subsection*{Acknowledgments}
I thank Andries Brouwer, Edwin van Dam, Alexander Gavrilyuk, Chris Godsil,
Aleksandar Jurisic, and Jack Koolen for their interest in this work and
their comments.

\end{document}